\documentclass[11pt]{amsart}
\usepackage[utf8]{inputenc}

\usepackage{amsmath}
\usepackage{amsthm}
\usepackage{amssymb}
\usepackage{mathtools}
\usepackage{bbm}
\usepackage{graphicx}
\usepackage{verbatim}
\usepackage{zref-xr}
\usepackage[toc,page]{appendix}

\usepackage[colorinlistoftodos]{todonotes}

\usepackage{hyperref}
\hypersetup{colorlinks,allcolors=blue}

\usepackage{cleveref}

\crefformat{section}{\S#2#1#3} 
\crefformat{subsection}{\S#2#1#3}
\crefformat{subsubsection}{\S#2#1#3}

\usepackage{tikz-cd}
\usepackage{comment}

\usepackage{biblatex}
\theoremstyle{plain}
\newtheorem{thm}{Theorem}[section]
\newtheorem{lemma}[thm]{Lemma}
\newtheorem{prop}[thm]{Proposition}
\newtheorem{cor}[thm]{Corollary}

\theoremstyle{definition}
\newtheorem*{definition}{Definition}

\theoremstyle{remark}
\newtheorem{remark}[thm]{Remark}
\newtheorem*{thank}{{\bf Acknowledgments}}
\newtheorem*{example}{Example}

\def\image{\text{Im }}

\def\diag{\text{diag}}
\def\Hom{\text{Hom}}

\def\det{\text{det}}

\def\Gm{\mathbb{G}_m}

\def\R{\mathbb{R}}
\def\Q{\mathbb{Q}}
\def\Z{\mathbb{Z}}

\def\C{\mathbb{C}}
\def\N{\mathbb{N}}

\def\X{\times}

\def\GL{\text{GL}}
\def\SL{\text{SL}}

\def\Pr{\text{Pr }}
\def\hw{\text{highest weight }}
\def\X{\times}
\def\1{\mathbbm{1}}

\def\ch{\text{characteristic}}

\def\<{\langle}
\def\>{\rangle}

\def\id{{\rm Id}}

\def\TS{\text{TS} }
\newcommand{\bigslant}[2]{{\raisebox{.2em}{$#1$}\left/\raisebox{-.2em}{$#2$}\right.}}
\def\l{\lambda}
\def\xx{\otimes}

\usepackage{subfiles}

\title{Asymptotic growth of trivial summands in tensor powers}

\author{Nai-Heng Sheu}
\address{ Department of Mathematics, Indiana University, Bloomington, IN 47405,
 U.S.A.}
\email{naihsheu@iu.edu}

\begin{document}

\begin{abstract}
    Given a finite-dimensional representation $V$ over an algebraically closed field of an abstract group $G$, we consider the number of the trivial summand counted with multiplicity in the direct sum decomposition of $V^{\xx n}$. We give necessary and sufficient conditions when the field is of characteristic $0$ and when the field is of characteristic $p$ so that $(V^{\xx n})_n$ has a subsequence $(V^{\xx n_k})_k$ such that $V^{\xx n_k}$ contains enough trivial summands when $k$ is sufficiently large.
\end{abstract}

\maketitle

\section{introduction}

Given $G$ a group and $V$ a representation of $G$, one may consider the sequence $(V^{\xx n})_n$. In \cite{coulembier2024growth}, Coulembier, Ostrik, and Tubbenhauer consider the number of indecomposable summands in $V^{\xx n}$ counted with multiplicity, denote this number by $b^{G, V}_n$, and show that $\lim_{n \to \infty} \sqrt[n]{{b^{G, V}_n}}=\dim V$. In \cite{larsen2024boundsmathrmsl2indecomposablestensorpowers}, Larsen observes that the growth rates of the multiplicity of summands with low highest weight where $G = \SL_2$ over an algebraically closed field of characteristic $2$, and where $V$ is the natural $2$-dimensional representation of $G$, are much slower than that in the characteristic zero setting, but still have an exponential growth. On the other hand, as a result of the discussion of the fusion graph in \cite{larsen2024boundsmathrmsl2indecomposablestensorpowers}, the multiplicity of trivial summands in the representation $V^{\xx n}$ where $V$ is the vector space of dimension $2$ over an algebraically closed field of characteristic $2$ as the natural representation of $\SL(V)$ is $0$ when $n>0$. Therefore, the exponential growth rate $d(\SL(V), V)$, defined in the next paragraph, is 0. This result separates the behaviors in the characteristic zero case and in the positive characteristic case and illustrates the main theorems in this paper.

Let $\rho: G \to \GL(V)$ be a finite-dimensional representation of $G$ over an algebraically closed field $K$. Let $\TS_G(V)$ denote the number of trivial summands in $V$ when $V$ is written as a direct sum of indecomposable representations of $G$. We are interested in the asymptotic behavior of $\TS_G(V^{\xx n})^{1/n}$. The limit $\lim_{n \to \infty} \TS_G(V^{\xx n})^{1/n}$ does not always exist but $\limsup_{n \to \infty}\TS_G(V^{\xx n})^{1/n}$ always exists. Denote $\limsup_{n \to \infty}\TS_G(V^{\xx n})^{1/n}$ by $d(G, V)$. Since $\TS_G(V^{\xx n})^{1/n} \le \dim V$,  $d(G, V)\le \dim V$. 

Keep the notations and assumptions in the previous paragraph, our main results as corollaries of Theorem \ref{MainTheorem} and Theorem \ref{chi finite first time} are the following:

\begin{thm} \label{main_char_zero}
    When $K$ is of characteristic zero, $d(G, V)=\dim V$ if and only if the determinant map takes finitely many values on $\rho(G).$ 
\end{thm}

\begin{thm} \label{main_char_p}
        When $K$ is of positive characteristic, $d(G, V)$ is equal to $\dim V$ if and only if the determinant map takes finitely many values on $\rho(G)$ and $H=\{e\}$ where $H$ is the set of unipotent elements of the Zariski closure of $\rho(G)$ in $\GL(V)$. 
\end{thm}

In \cref{sec: notations}, we set up the notations and the conventions that will be used throughout the entire article. In \cref{sec: general}, we reduce the condition that $G$ is an abstract group to the condition that 
$G$ is a Zariski closed subgroup of $\GL(V)$ for some finite-dimensional vector space $V$. We also discuss the results that are independent of the characteristic of the field $K$. In \cref{sec: char 0}, we discuss the results where $K$ is of characteristic zero. We start with the special linear groups and extend the result to the representations on which the determinant map takes finitely many different values. In \cref{sec: char p}, we discuss the results when $K$ is of positive characteristic. In \cref{appendix: decomp}, we give a formula for the decompositions of tensor products of modular representations of $\Z/p\Z$ and in \cref{appendix: regular}, we give some observations on the regular representations over a field of \ch\ zero. We use the results from \cref{appendix: decomp} and \ref{appendix: regular} in \cref{sec: char p}.

\begin{thank}
    The author would like to thank Michael Larsen for discussions, comments, and  support, and Noah Snyder for the suggestion and discussions. 
    
\end{thank}


\section{notations and conventions} \label{sec: notations}

Let $G$ be an arbitrary group. In this article, all representations we consider are finite-dimensional over $K$ where $K$ is an algebraically closed field.

Let $\rho: G \to \GL(V)$ be a  representation of $G$ over $K$.  Let $\TS_G(V)$ denote the number of trivial summands in $V$ counted with multiplicity when we  write $V$ as a direct sum of indecomposable representations of $G$. If $H$ is a subgroup of $G$, then $\TS_H(V)$ is the number of trivial summands in $V$ as a restricted representation of $H$. It is clear that $\TS_H(V)\ge \TS_G(V)$.  When there is no ambiguity,  we may drop the subscript $G$. We denote $\limsup_{n \to \infty}\TS_G(V^{\xx n})^{1/n}$ by $d(G, V)$.

By a slight abuse of terminology which should not cause confusion, we identify a closed subvariety of $\GL_m(K)$ with its set of points over $K$. By saying the Zariski closure of a subset of $\GL_m(K)$, we mean the smallest closed subvariety containing that set, and by saying that $G$ is an algebraic group we mean $G$ is a subgroup and a closed subvariety of $\GL_m(K)$ for some $m$.


\section{general case} \label{sec: general}

Given an abstract group $G$ and $\rho: G \to \GL(V)$ a finite-dimensional representation of $G$. The following proposition allows us to assume $G$ is a closed subgroup of $\GL(V)$ and $V$ is a representation of $G$ in the natural way.

 \begin{prop}
     Let $\rho: G \to \GL(V)$ be a finite-dimensional representation of $G$ and $\overline{G}$ be the Zariski closure of $\rho(G)$ in $\GL(V)$. Then $$d(G, V)=d( \overline{G}, V).$$
 \end{prop}

\begin{proof}
Let $V$ be an $m$-dimensional vector space over a field $K$. Let $W$ be a $d$-dimensional subspace of $V$. Let $\{v_1,\cdots,v_m \}$ be an ordered and linearly independent set of $V$ where $\{v_i\}_{i=1}^d$ forms a basis of $W$. After fixing a basis of $V$, the set $\{v_1,\cdots,v_m\}$ gives an element of $\GL_m(K)$. Each element of the subgroup $\GL_d(K)\times I_{m-d}$ acts on $\GL_m(K)$ via the right multiplication and sends the set $\{v_1,\cdots,v_m \}$ to $\{v'_1,\cdots,v'_d,v_{d+1}, \cdots, v_m\}$ for some $\{v'_i\}_{i=1}^d$ which forms a basis of $W$. On the other hand, let $\{v'_i\}_{i=1}^d$ be a basis of $W$; there exists a unique element of $\GL_d(K)$ sending $\{v_i\}_{i=1}^d$ to $\{v'_i\}_{i=1}^d$. Therefore, a pair $(W, \{v_{d+1}, \cdots, v_m\})$ is identified as a single orbit of the invertible matrix, which is given by the set $\{v_1,\cdots,v_m\}$,  under the action of $\GL_d(K)$, hence, an element of the quotient variety $\GL_m(K)/\GL_d(K)$. 

An element $x\in \GL_m(K)/\GL_d(K)$ gives a morphism $$\phi_x: \GL_m(K) \to \GL_m(K)/\GL_d(K)$$ by $g \mapsto gx$. Suppose $x$ is identified as $(W, \{v_{d+1}, \cdots, v_m\})$. The preimage of $x$ under $\phi_x$ is a closed subgroup of $\GL_m(K)$ and it consists of the elements in $\GL_m(K)$ that stabilize $W$ and fix the vector $v_i$ for $d<i \le m$; hence, it is the isotropy subgroup of the pair $(W, \{v_{d+1}, \cdots, v_m\})$ in $\GL_m(K)$.

Now suppose $\rho: G \to \GL(V)$ is a representation of $G$ and $v$ is a nonzero vector in a trivial summand of $V$ as a representation of $G$. Let $W$ be a $G$-subepresentation of $V$ that is complement of $Kv$ in $V$. Since $\rho(G)$ fixes $v$ and stabilizes $W$, $\rho(G)$ fixes the pair $(W, v)$. Therefore, the Zariski closure of $\rho(G)$ in $\GL(V)$ is contained in the isotropy subgroup of $(W, v)$. Denote the Zariski closure of $\rho(G)$ by $\overline{G}$. It shows that $\TS_{\overline{G}}(V)\ge \TS_G(V)$. 
On the other hand, since $\rho(G)$ is a subgroup of $\overline{G}$, one has $\TS_G(V) \ge \TS_{\overline{G}}(V)$. Therefore, $\TS_G(V)= \TS_{\overline{G}}(V)$.

Every natural number $n$ gives a morphism $\tau_n: \GL(V)\to \GL(V^{\xx n})$ and it induces a morphism $\rho_n: G \to \GL(V^{\xx n})$. By a similar argument, one has $\TS_G(V^{\xx n })= \TS_{\overline{G}}(V^{\xx n}).$ 
\end{proof}

\begin{thm} \label{chi infinite}
    Let $V$ be an $m$-dimensional vector space over $K$ where $K$ is algebraically closed. Let $G$ be a closed subgroup of $\GL(V)$. 
    If $\image \left.\det\right|_{G^\circ}$ is of infinite order, then $d(G, V)<m$.
\end{thm}

\begin{proof}
    We may assume $G$ is connected since the connected component $G^\circ$ is of finite index and hence $\image \left.\det\right|_{G^\circ}$ is of infinite order. By \cite[Theorem 11.10]{borel2012linear}, $G$ is a union of Borel subgroups. Since any two Borel subgroups are conjugate to each other, and the determinant is determined up to conjugation, we may assume $\image \left.\det\right|_B$ is of infinite order. For any Borel subgroup $B$, $B=T \ltimes U$, where $T$ is a maximal torus and $U$ is the unipotent subgroup of $B$, the determinant of a unipotent element is $1$; therefore, the dimension of $T$ is at least $1$ and $\det(T) \supsetneq \{1\}$. Hence, $T$ contains a $1$-dimensional torus $S$. Since $K$ is algebraically closed, we may assume 
\begin{multline*}
S= \{\diag(t^{k_1}, \cdots, t^{k_m})\mid t \in K \ \text{for some }k_1, \cdots, k_m \\
\text{ that sum up to a nonzero number}\}.
\end{multline*}

    Since $d(S, V) \ge d(T, V)\ge d(G, V)$, it suffices to show that $d(S, V)< m$. Replace $T$ by $S$.

    Write $V=K v_1\oplus \cdots \oplus K v_m$ where $\diag(t^{k_1}, \cdots, t^{k_m})  v_i= t^{k_i} v_i$ and $v_i$ is of weight $k_i$ for the torus $T$. Let $v_{a_1}\xx  \cdots \xx v_{a_n}$ be an element of $V^{\xx n}$ and $n_i$ be the number of times that $i$ occurs in $\{a_k\}_{k=1}^n$. Then $v_{a_1}\xx \cdots \xx v_{a_n}$ is of weight $\sum_{i=1}^m n_i k_i$ for the torus $T$.

    Consider a random walk in $\Z$ starting from the origin. For each step, there are $m$ possible movements $k_i$ with probability $1/m$. Let $k=\sum_{i=1}^m k_i$ where $k_i$ is given by $S$; therefore, $k>0$.  Let $X_i$ be the random variable of $i$-th step movements. The expected value of $X_i$ is $k/m$. The random variables $X_i$ are i.i.d. Let $S_n=\sum_{i=1}^n X_i$. The expected value of $S_n$ is $nk/m$.

    Since $\TS_T(V^{\xx n})\le m^n \Pr[S_n=0]$ and $$\Pr[S_n=0] \le \Pr[S_n-  \frac{nk}{m} < -\frac{nk}{m+1}],$$ it suffices to show that $\Pr[S_n-  \frac{nk}{m} < -\frac{nk}{m+1}]$ is sufficiently small. 

    Let $Y_i= k/m-X_i$ and $Z_n=\sum_{i=1}^n Y_i$. To show that $\Pr[S_n-  \frac{nk}{m} < -\frac{nk}{m+1}]$ is sufficiently small is equivalent to show that $\Pr[Z_n > \frac{nk}{m+1}]$ is sufficiently small.  

    Let $b$ be a positive number greater than all the possible values of $Y_i$. Let $v=\sum_{i=1}^n E[Y_i^2]$.
    By Bernstein's inequality \cite[Equation 2.10]{boucheron2013concentration}, $$\Pr[Z_n > t] < e^{-t^2/2(v+bt/3)}.$$
    
    Let $t=nk/(m+1)$, then $t^2/2(v+bt/3)>Cn$ for some $C>0$ when $n>N$ for some $N>0$. Therefore, $$\Pr[Z_n > t] < e^{-Cn}$$ for some $C>0$, and  $$ m^n \Pr[Z_n > t] < m^n e^{-Cn}.$$ Hence, $\TS_T(V^{\xx n})^{1/n}< m e^{-C}$ for some $C>0$  when $n>N$. Therefore, $d(G, V)< d(T, V) < m.$
    
\end{proof}


\section{Characteristic $0$ case} \label{sec: char 0}

\subsection{Asymptotic behavior of the special linear group}

Let $V$ be an $m$-dimensional vector space over $K$ where $K$ is algebraically closed and of \ch\ zero. We are interested in the asymptotic behavior of $\limsup_{n \to \infty} \TS_G(V^{\otimes n})^{1/n}$ where $G$ is $\SL(V)$. We begin by fixing notation and stating several results from weight theory that will be used in the proof of Theorem \ref{theoremForSL}.

Let $e_i$ be the $i$-th coordinate vector of $\Z^m$. Express the weights of $G$ as $[\l]$ where $\lambda=(\lambda_1, \ldots, \lambda_m)$, i.e. a weight is an element of $\bigslant{\Z^ m}{\Z(e_1 +\cdots +e_m)}$. A weight $[\lambda]$ is dominant if and only if $\lambda_1 \ge \lambda _2 \ge \cdots \ge \lambda_m$. Moreover, one can choose a representative of $[\l]$ such that every coordinate is non-negative. Hence, a dominant weight corresponds to one $m$-part partition (a weakly decreasing $m$-tuple of non-negative integers) of some integer.

Given a dominant weight $[\l]$, we denote the irreducible representation of \hw $[\l]$ by $V_\l$.
The natural representation $V$ is isomorphic to $V_{e_1}$. We want to know the decomposition of $V_\l \otimes V_{e_1}$.

Let $\Tilde{\l}$ be a representative of $[\l]$ as a partition. By the Littlewood-Richardson rule (or Pieri's formula for this special case),

\[ V_\l \otimes V \cong \bigoplus_\mu  V_\mu,\]
where the sum is over all partitions $\mu$ obtained by adding 1 to one coordinate of $\Tilde{\l}$. 

As a result and by induction on $n$,  $V_\l$ is a summand of $V^{\xx n}$ if and only if $[\l]$ has a representative that is a partition of $n$.\\

Let $\rho$ be half the sum of the positive roots, and $\l=(\l_1, \ldots, \l_m)$ be a partition of $n$. By the Weyl dimension formula,

\[ \dim V_\l =\prod_{\alpha \in R^+}\frac{  (\l + \rho, \alpha) }
{ (\rho, \alpha)  }=\prod_{1 \le i<j \le m}\frac{(\l_i-\l_j)+(j-i)}{j-i}\le \prod_{1 \le i<j \le m}\frac{n+(j-i)}{j-i} \]
where $R^+$ is the set of positive roots. Therefore, $\dim V_\l \le P(n)$ for some polynomial $P(x)$ of degree $\lvert R^+\rvert =\frac{m(m-1)}{2}.$
\\


Any weight $[\l]$ of $V^{\xx n}$ has a unique representative whose coordinates sum up to $n$.

\begin{definition}
    Given an $m$-tuple $x=(x_1, \ldots x_m)$ of non-negative numbers such that $\sum_{i=1}^m x_i=n$, we say $x$ is \emph{close to the mean} if $\vert x_i - n/m \vert < n^{2/3}$ for all $i$; otherwise, we say $x$ is not close to the mean. We say a weight $[\l]$ of $V^{\xx n}$ is \emph{close to the mean} if the representative whose coordinates sum up to $n$ is close to the mean; otherwise, we say $[\l]$ is not close to the mean. 
\end{definition}

In the following, $\l$ always denotes an $m$-tuple with non-negative integers that sum up to $n$.

\begin{lemma}\label{dualrep}
    Let $\l$ be a partition of $n$ such that $\l$ is close to the mean. The dual representation of $V_\l$ is a direct summand of $V^{\xx n'}$ for some $n'<mn^{2/3}.$ 
\end{lemma}

\begin{proof}
    Given an arbitrary partition $\l=(\l_1, \ldots, \l_m)$ of $n$, the dual representation $V_\l^*$ has highest weight $[\mu]$ where if $w_0$ is the longest element of the Weyl group of $G$, then \begin{align*}
    [\mu]&=[-w_0(\l_1, \ldots, \l_m)] \\
&=[-(\l_m, \l_{m-1}, \ldots, \l_2, \l_1)]. 
\end{align*}

Adding $[\l_1 (1, \ldots, 1)]$ to $[\mu]$, then $$[\mu]=[(\l_1-\l_m, \l_1-\l_{m-1}, \ldots, \l_1-\l_2, 0)].$$ Then $$(m-1)\l_1-(\l_m+\l_{m-1}+\cdots + \l_2)=m\l_1 -(\l_m+\l_{m-1}+\cdots + \l_2+\l_1)=m\l_1 - n.$$
 Especially, when $\l$ is close to the mean, then $\l_1 - n/m<n^{2/3}$. Hence, $m\l_1 - n<mn^{2/3}$. Therefore, $V_\mu$ is a direct summand of $V^{\xx n'}$ for some $n'<mn^{2/3}$.
\end{proof}

\begin{lemma}\label{V_l}
    When $n$ is large enough, there is a direct summand $V_\l$ of $V^{\otimes n}$ where $\l$ is close to the mean and the multiplicity of $V_\l$ in $V^{\xx n}$ is greater than $Cm^n/n^k$ for some $C$ and $k$ depending only on $m$.
\end{lemma}

\begin{proof}
    Let $\{ v_1, \ldots, v_m\}$ be the standard basis of $V$; hence $v_i$ is of weight $[e_i]$. Then \begin{equation} \label{form of basis}
    \{v_{a_1}\xx v_{a_2} \xx \cdots \xx v_{a_n} \mid (a_1, a_2, \ldots, a_n) \in \{1, 2, \ldots, m\}^{\X n} \}
\end{equation}
 is a basis of $V^{\xx n}$. For each $v_{a_1}\xx  \cdots \xx v_{a_n}$, let $n_i$ be the number of times that $i$ occurs in $\{a_k\}_{k=1}^n$; then $v_{a_1}  \xx \cdots \xx v_{a_n}$ is of weight $[(n_1, \ldots, n_m)]$, and $[(n_1, \ldots, n_m)]$ is of multiplicity $\frac{n!}{\prod n_i!}$.

Consider a random walk in $\Z^m$ starting from $(0, \ldots, 0)$. For each step, there are $m$ possible directions $e_i$ with probability $1/m$. After $n$ steps, the probability of being at $(\l_1, \ldots, \l_m)$ is given by the multinomial coefficient $\frac{n!}{\prod \l_i !\ }$.

Let $X_i$ be the random variable of the first coordinate of the $i$-th step direction. For each $X_i$, the possible outcomes are $1$ or $0$, with probability $1/m$ and $(m-1)/m$, respectively. The random variables $X_i$ are i.i.d. The value of $\sum_{i=1}^n E[X_i^2]$ is $n/m$. Let $S_n=\sum_{i=1}^n X_i$. The expected value of $S_n$ is $n/m$. 

By Bernstein's inequality,
$$\Pr\Bigl[\vert S_n-\frac{n}{m} \vert >t\Bigr]< 2e^{-t^2/2(n/m+t/3)}. $$

Let $t=n^{2/3}$. Then \[\frac{t^2}{2(n/m+t/3)}=\frac{n^{4/3}}{2(n/m+n^{2/3}/3)}>Cn^{\frac{1}{3}}\] for some $C$ when $n>N$ for some $N>0$.

Hence, 

\begin{align*}
    \Pr\Bigl[ \vert S_n-\frac{n}{m}\vert > n^{2/3}\Bigr] &< 2e^{-Cn^{1/3}}, \\
    \Pr\bigl[ (\lambda_1, \ldots, \lambda_m) \text{ is not close to the mean}\bigr] &< 2me^{-Cn^{1/3}},\\
    \Pr\bigl[ (\lambda_1, \ldots, \lambda_m) \text{ is  close to the mean}\bigr] &> 1- 2me^{-Cn^{1/3}}.
\end{align*}

 The total multiplicity of weights of $V^{\xx n}$ that are close to the mean is larger than $m^n\bigl(1-2me^{-Cn^{1/3}}\bigr) > Dm^n$ for some $D>0$ whenever $n>N$ for some $N$. Let $A$ be the sum of the dimensions of summands $V_\l$ of $V^{\xx n}$ with multiplicity, where $\l$ is close to the mean. Since the total multiplicity of weights not close to the mean is smaller than $m^n\cdot 2me^{-Cn^{1/3}}$, and since by the Weyl dimension formula, the dimension of each irreducible summand of $V^{\xx n}$ is bounded by $P(n)$, it follows that the sum of dimensions of $V_\mu$ with multiplicity, where $V_\mu$ is a summand of $V^{\xx n}$ and $\mu$ is not close to the mean, is $\mathcal{O}\Bigl(m^{n+1}e^{-Cn^{1/3}} n^{m(m-1)/2}\Bigr)$. Therefore, $A > Dm^n$ for some $D>0$ when $n$ is large enough.

Let $P_m(n)$ be the number of partitions of $n$ with at most $m$ parts. The number $P_m(n)$ is bounded by the number of the ways of putting $n$ undistinguished balls into $m$ bins, which is the value of a polynomial in $n$. Hence, $P_m(n)<Cn^k$ for some $C$ and $k$ when $n$ is large enough. Therefore, the number of isomorphism classes of $V_\l$ where $\l$ is close to the mean is smaller than $P_m(n)$, and hence smaller than $Cn^k$.

It follows that there exists $V_\l$ where $\l$ is close to the mean, whose multiplicity in $V^{\xx n}$ is greater than $A/Cn^k> C'm^n/n^k$ for some $C'$ when $n$ is large enough.
\end{proof}

Now we are ready to prove the main theorem of this section.

\begin{thm}\label{theoremForSL}
    Let $V$ be the natural representation of $\SL(V)$. Then $$\limsup_{n \to \infty}\TS_{\SL(V)}(V^{\xx n})^{1/n}=\dim V=m.$$
\end{thm}

\begin{proof}

Let $V_\l$ be a representation provided by Lemma \ref{V_l} for some sufficiently large $n$. Therefore, the  multiplicity of $V_\l$ in $V^{\xx n}$ is greater than $Cm^n/n^k$ and $\l$ is close to the mean. By Lemma \ref{dualrep}, there is some $n'<mn^{2/3}$ such that $V^{\xx n'}$ contains $V_\l^*$. Therefore, $V^{\xx n} \xx V^{\xx n'}$ contains $V_\l \xx V_\l^*$ with multiplicity greater than $Cm^n/n^k$. Since the trivial representation is a summand of $W\xx W^*$ for any representation $W$, 
\[\TS_{\SL(V)} \bigl(V^{\otimes n+n'}\bigr)^{\frac{1}{n+n'}}>\bigl(Cm^n/n^k\bigr)^{\frac{1}{n+n'}}>\bigl(Cm^n/n^k\bigr)^{\frac{1}{n+mn^{2/3}}}.\] 
As $$
\lim_{n\to\infty} \bigl(Cm^n/n^k\bigr)^{\frac{1}{n+mn^{2/3}}}=\lim_{n \to \infty} m^{\frac{n}{n+mn^{2/3}}}=\lim_{n \to \infty} m^{\frac{1}{1+mn^{-1/3}}}=m,$$ it follows that there is a subsequence of $(\TS_{\SL(V)}(V^{\xx k}))_k$ such that $$\lim_{l \to \infty} \TS_{\SL(V)}\bigl(V^{\otimes k_l}\bigr)^{1/k_l}=m.$$ Hence, \[\limsup_{n \to \infty} \TS_{\SL(V)} \bigl(V^{\otimes n}\bigr)^{1/n} = m.\]

\end{proof}

\begin{remark} \label{ReduceToLimit}
As $n+n' = n +(m\lambda_1-n)=m\lambda_1$, there is a subsequence $\bigl(\TS_{\SL(V)}\bigl(V^{\xx m n_k}\bigr)\bigr)_k$ of $\TS_{\SL(V)} (V^{\xx mn})$ such that \[\lim_{k \to \infty} \TS_{\SL(V)} \bigl(V^{\otimes mn_k}\bigr)^{1/mn_k} = m. \] Therefore, \[ \limsup_{n \to \infty} \TS_{\SL(V)} \bigl(V^{\otimes mn}\bigr)^{1/mn} = m.
 \]

Let $a_l=\TS_{\SL(V)}\bigl(V^{\xx ml}\bigr)$. 
The representation $V^{\xx m}$ contains $\bigwedge^m V$, which is trivial since $G=\SL(V)$. Therefore, $a_l>0$. Moreover, by distribution of tensor product over direct sum, $a_{l+k}\ge a_l \cdot a_k$. By Fekete's subadditive lemma \cite{fekete1923verteilung}, $\lim_{n \to \infty} a_n^{1/n}$ exists. Hence, \[\lim_{n \to \infty} \TS_{\SL(V)} \bigl(V^{\otimes mn}\bigr)^{1/mn}= m.\]
\end{remark}

\subsection{Asymptotic behavior of an arbitrary group}

\begin{prop} \label{chi finite}
    Let $V$ be an $m$-dimensional vector space over $K$ where $K$ is algebraically closed. Let $G$ be a closed subgroup of $\GL(V)$. If $\image \left.\det\right|_G$ is of order $k$, then  $d(G, V)=m.$
\end{prop}
\begin{proof}
    Let $d$ be an arbitrary natural number and $$Z=\{\omega I \in \GL(V) \mid \omega \text{ is a } d \text{-th root of unity}\}.$$ The action of an element $\omega I$ of $Z$ on $V$ is the scalar multiplication by $\omega$. Therefore, $Z$ acts trivially on $V^{\xx d}.$ 

    Consider the subgroup $Z\SL(V)$ of $\GL(V)$.

    By Remark \ref{ReduceToLimit}, we have $\lim_{n \to \infty} \TS_{\SL(V)}\bigl(V^{\xx m n}\bigr)^{1/mn}=m$. Therefore, $$\lim_{n \to \infty} \TS_{\SL(V)}\bigl(V^{\xx m nd}\bigr)^{1/mnd}=m.$$

    Since $Z$ acts trivially on $V^{\xx mn d}$, we have 
    $$\TS_{Z\SL(V)}\bigl(V^{\xx mn d}\bigr)=\TS_{\SL(V)}\bigl(V^{\xx mn d}\bigr).$$
    Therefore, $d(Z\SL(V), V)=m$.

    Now suppose $G$ is a closed subgroup of $\GL(V)$, $\image\left.\det\right|_G$ is of order $k$ and $\dim V=m$. Then $\det (g)$ is a $k$-th root of unity. 
    Consider $$Z=\{ \omega I\in \GL(V) \mid \omega \text{ is an } mk \text{-th root of unity} \}.$$ Since $$g=\det (g)^{1/m} \bigl(\det (g)^{-1/m} g\bigr) \in Z \SL(V),$$ it follows that $G \subset Z\SL(V) \subset \GL(V)$. Hence $\TS_G(V^{\xx n}) \ge \TS_{Z\SL(V)}(V^{\xx n})$ for all $n$, and it follows that $$d(G, V)\ge d(Z\SL(V), V)=m.$$
\end{proof}

\begin{thm} \label{MainTheorem}
    Let $G$ be a closed subgroup of $\GL(V)$. Then $$d(G, V)=\dim V$$ if and only if $\image\left.\det\right|_G$ is of finite order. 
\end{thm}

\begin{proof}
    It follows from Propositions \ref{chi finite} and \ref{chi infinite}.
\end{proof}

Theorem \ref{main_char_zero} follows from Theorem \ref{MainTheorem} and the following observation: given an abstract group $G$ in $\GL(V)$, the determinant map takes finitely many distinct values on $G$ if and only if it takes finitely many distinct values on $\overline{G}$ where $\overline{G}$ is the Zariski closure of $G$ in $\GL(V)$.

\section{Characteristic $p$ case} \label{sec: char p}

\subsection{Asymptotic behavior of the cyclic group of order $p$}
Let $$G=\Z/p\Z =\langle g \rangle ,$$ $\pi:G \to \GL(V)$, $V$ a representation of $G$ over $K$, where $K$ is a field of characteristic $p$. Since $g^p=e$, it follows that $(\pi(g)- \pi (e))^p=0$. Hence, the only eigenvalue of $\pi(g)$ is $1$, and each Jordan block is of size smaller than or equal to $p$. Suppose $V$ is indecomposable; then $\pi(g)$ has only one Jordan block. Hence, $1 \le \dim V \le p$. Let $V_i$ denote the indecomposable representation of the cyclic group of dimension $i+1$.

Let $R(G)$ be the representation ring over $K$. Abusing notation, we use $V_i$ to denote both the representation $V_i$ and its isomorphism class. Therefore, $R(G)=\Z V_0 \oplus \cdots \oplus \Z V_{p-1}$. Given a homomorphism of additive groups $S: R(G) \to R(G)$, let $[S]$ denote the matrix of the map $S$ with respect to the ordered basis $\{V_0, \cdots, V_{p-1}\}$; given $W\in R(G)$, let $[W]$ denote the matrix of $W$.

Let $R^+(G)$ be the subset of $R(G)$ consisting of isomorphism classes of representations of positive dimension. Define a map $Q:R^+(G) \to \Q^p$ that sends the class of a representation $V$ to a probability vector (i.e., a vector with nonnegative coordinates that sum up to 1) as follows \[V=a_0 V_0+\cdots + a_{p-1}V_{p-1} \longmapsto \frac{1}{\dim V} \begin{bmatrix}
    a_0\\
    2a_1\\
    \vdots \\
    pa_{p-1}
\end{bmatrix}. \]
In other words, the vector $Q(V)$ consists of the ratio of the dimension of $a_iV_i$ to the dimension of $V$.

Given a homomorphism of additive semigroups $S: R(G)^+ \to R(G)^+$, define $P(S)=\begin{bmatrix}
    w_0 &w_1 &\cdots &w_{p-1}
\end{bmatrix}\in M_p(\R)$, where the column vector $w_i=Q(S(V_i))$.

\begin{example}

Let $p=2$ and $S: R(G) \to R(G)$ a homomorphism of additive groups such that     $[S]=\begin{bmatrix}
        2 &1\\
        0 &1
    \end{bmatrix}.$ Then 
    $[S^2]=\begin{bmatrix}
        4 &3\\
        0 &1
    \end{bmatrix},$    
    $P(S)=\begin{bmatrix}
        1 &1/3\\
        0 &2/3
    \end{bmatrix},$ and
    $P(S^2)=\begin{bmatrix}
        1 &3/5\\
        0 &2/5
    \end{bmatrix} \neq P(S)^2.   
    $ 
    
\end{example}
 
\begin{lemma}\label{stst}
   Let $S_1,\ S_2$ be two endomorphisms of $R(G)$ as an additive group. If there is a representation $V_{S_i}$ such that $S_i$ sends $V$ to $V\xx V_{S_i}$, then $P(S_1)\cdot P(S_2)=P(S_1 \circ S_2)$.
\end{lemma}

\begin{proof}
    Given $S_i$ and  $V_{S_i}$ satisfying the condition, let $d_i= \dim V_{S_i}$. Let $D=\text{diag}(1, 2, \cdots, p)$.

    Then $[S_i]=\begin{bmatrix}
        v_0 & v_1 &\cdots &v_{p-1}
    \end{bmatrix}$ where $v_j=[S_i(V_j)]$. The sum of the coordinates of $Dv_j$ is $(j+1)\cdot d_i$.
    
    Hence,$$P(S_i)=\begin{bmatrix}
        \frac{Dv_0}{d_i} &\frac{Dv_1}{2\cdot d_i} &\cdots &\frac{Dv_{p-1}}{p\cdot d_i}\end{bmatrix}        =D\cdot [S_i] \cdot D^{-1} \cdot d_i^{-1}. $$

    Then \begin{align*}
        P(S_1)\cdot P(S_2)&=D\cdot [S_1] \cdot D^{-1} \cdot d_1^{-1} D\cdot [S_2] \cdot D^{-1} \cdot d_2^{-1}\\
        &=D\cdot [S_1] \cdot [S_2] \cdot D^{-1}\cdot d_1^{-1}d_2^{-1}\\
        &=D\cdot [S_1 \circ S_2]\cdot D^{-1}\cdot (d_1 d_2)^{-1}=P(S_1 \circ S_2).
    \end{align*}
\end{proof}

\begin{prop}\label{d(G, V_1)}
    Let $V_1$ be the indecomposable representation of $G$ of dimension 2. The number $d(G, V_1) < 2.$
\end{prop}

\begin{proof}
    If $p=2$, it is obvious since $V_1^{\xx n}= V_1^{\oplus 2^{n-1} }.$ Therefore, we assume $p>2$ from now on. 

Let $T:R(G)\to R(G)$ send $V$ to $V_1\otimes V$. Then $P(T)^n=P(T^n)$ by Lemma \ref{stst}. Since $V_1^{\xx p-1}$ has $V_{p-1}$ as a direct summand, it follows that for all $j$, $V_1^{\xx p-1}\xx V_j$ has $V_{p-1}$ as a direct summand by Proposition \ref{decomVp-1}; therefore, all entries in the $p$-th row of $P(T^{p-1})$ are nonzero. Let $\epsilon$ be a positive number smaller than the minimum value of those entries. Let $s_i$ be the sum of the first $p-1$ coordinates of the $i$-th column of $P(T^{p-1}).$ Since the $p$-th coordinate of the $i$-th column of $P(T^{p-1})$ is larger than $\epsilon$, $s_i<1-\epsilon$ for all $i$. Since $V_i \xx V_{p-1}=(i+1) V_{p-1}$ for any $i$ by Proposition \ref{decomVp-1}, the $p$-th column of $P(T^{p-1})$ is the vector with $1$ in the $p$-th coordinate and $0$ in the other coordinates, hence $s_p=0$. 

Let $w=[
    a_1 \  \cdots \  a_p
]^T$ be an arbitrary probability vector, and write $P(T^{p-1})w=[b_1 \  \cdots \  b_p]^T$. Then $$b_1+\cdots +b_{p-1}=\sum^p_{i=1} a_i \cdot s_i=\sum ^{p-1}_{i=1} a_i\cdot s_i < (\sum^{p-1}_{i=1}a_i)(1-\epsilon) < 1-\epsilon.$$ Especially, if $\sum ^{p-1}_{i=1} a_i < (1-\epsilon)^n$, then $\sum ^{p-1}_{i=1} b_i <(1-\epsilon)^{n+1}$. Let $w_0=w$ and $w_{n+1}=P(T^{p-1})w_n$. By induction on $n$, the sum of the first $(p-1)$-th coordinates of $w_n$ is smaller than $(1-\epsilon)^n$.

Take $w_0=[1 \  0 \ \cdots \  0]^T$. Each entry of $w_n$ is the ratio of the dimension of $V_i$ with multiplicity to the dimension of $V_1^{\xx (p-1)n}$, which is $2^{(p-1)n}$. Since the sum of the first $(p-1)$-th coordinates of $w_n$ is smaller than $(1-\epsilon)^n$, the total dimension of the factors $V_0, V_1, \cdots, V_{p-2}$  with multiplicity in $V_1^{\xx (p-1)n}$ is smaller than $2^{(p-1)n} (1-\epsilon)^n$; therefore, when  write $$V_1^{\xx (p-1)n}=a_0 V_0 \oplus a_1 V_1 \oplus \cdots \oplus a_{p-1}V_{p-1},$$ one has \begin{equation} \label{boundofV0}
    a_0<2^{(p-1)n} (1-\epsilon)^n.
\end{equation}

Assume $p>2$. Since $V_1\xx V_1=V_0\oplus V_2$, it follows that $\TS_G\bigl(V_1^{\xx 2}\bigr)>0$ and $\TS_G\bigl(V_1^{\xx 2n}\bigr)>0$. Since $\TS_G\bigl(V_1^{ \xx 2n+2m}\bigr) \ge \TS_G\bigl(V_1^{\xx 2n}\bigr)\cdot \TS_G\bigl(V_1^{\xx 2m}\bigr),$  by Fekete's subadditive lemma, $\lim_{n \to \infty} \TS_G\bigl(V_1^{\xx 2n}\bigr)^{1/2n}$ exists.

Consider the subadditive sequence $c_n=- \log \TS_G\bigl(V_1^{\xx 2n}\bigr)$. It is well-defined since $\TS_G \bigl(V_1^{\xx 2n}\bigr)>0$. By Fekete's lemma, $\lim_{n \to \infty} c_n/n$ exists so $\lim_{n \to \infty} c_n/2n$ exists. Therefore, $\lim_{n \to \infty} \TS_G \bigl(V_1^{\xx 2n}\bigr)^{-1/2n}$ exists and it is not equal to $0$. Therefore, $\lim_{n \to \infty} \TS_G \bigl(V_1^{\xx 2n}\bigr)^{1/2n}$ exists.

Since $\TS_G\Bigl(V_1^{\xx (p-1)n}\Bigr)^{1/(p-1)n}$ is a subsequence of $\TS_G\bigl(V_1^{\xx 2n}\bigr)^{1/2n}$, it follows that $\lim_{n \to \infty
} \TS_G\Bigl(V_1^{\xx (p-1)n}\Bigr)^{1/(p-1)n}$ exists. By (\ref{boundofV0}), $$\lim_{n \to \infty} \TS_G\Bigl(V_1^{\xx (p-1)n}\Bigr)^{1/(p-1)n}<2. $$ Therefore, $$\lim_{n \to \infty} \TS_G\bigl(V_1^{\xx 2n}\bigr)^{1/2n}<2.$$ For any $n$, one has $\TS_G\bigl(V_1^{\xx n}\bigr)^2 \le \TS_G\bigl(V_1^{\xx 2n}\bigr)$; hence, $$\limsup_{n\to \infty} \TS_G\bigl(V_1^{\xx n}\bigr)^{1/n} \le \lim_{n \to \infty} \TS_G \bigl(V_1^{\xx 2n}\bigr)^{1/2n} < 2.$$
\end{proof}

\begin{thm} \label{d(P, V)}
    Suppose $V$ is a representation of $G$ with some direct summand $V_i$ where $i\neq 0$.     Then $d(G, V) < \dim V.$
\end{thm}

\begin{proof}
    As the case $V=V_1$ is proved, we may assume $V$ is isomorphic to $V_i$ for some $i>1$ or $V$ is decomposable and at least one of the summands is nontrivial.

In the proof of Proposition \ref{d(G, V_1)}, we use that (1): $V_1^{\xx (p-1)} \xx V_i$ has $V_{p-1}$ as a direct summand for any $i$, and (2): $V_1\xx V_1$ has $V_0$ as a direct summand. 

Hence, given an arbitrary $V$ with these two properties (a): $V^{\xx (p-1)} \xx V_i$ has $V_{p-1}$ as a direct summand for any $i$, and (b): $V\xx V$ has $V_0$ as a direct summand, the argument applies. Therefore, $\lim_{n \to \infty}\TS_G(V^{\xx 2n} )^{1/2n}$ exists as in Proposition \ref{d(G, V_1)} and this number is smaller than $\dim V,$ so $$d(G, V) < \dim V.$$

It remains to show that $V$ has property (a) and (b).

If $V=V_m$ for some $m>1$, by the formula in Proposition \ref{decomVm} one has: \begin{equation*}
    V_m \xx V_i=
    \begin{cases}
        V_{\lvert m-i \rvert} \oplus V_{\lvert m-i+2 \rvert} \oplus \cdots \oplus V_{\lvert m+i \rvert},\ \text{when } m+i \le p-1, \\

        dV_{p-1} \oplus V_{m-d}\xx V_{i-d},\ \text{when } m+i=p-2+d \ \text{and}\ d\ge 1.
    \end{cases}
\end{equation*}
This shows that $V$ has properties (a) and (b).

Suppose $V$ is a representation of $G$ and $V$ has a direct summand $V_i$ that is not isomorphic to $V_0$. Since $V^{\xx k}$ contains $V_i^{\xx k}$ as a direct summand for any $k$, $V$ satisfies properties (a) and (b). 
\end{proof}

\subsection{Asymptotic behavior of an arbitrary group}

    Let $K$ be an algebraically closed field of \ch\ $p$. Suppose $G$ is a closed subgroup of $\GL(V)$ where $V$ is an $m$-dimensional vector space over $K$. Let $H$ be the subset of $G$ that contains all unipotent elements.  
\begin{prop} \label{char_p_easy_direction}
    If  $d(G, V)=\dim V$, then $H=\{e\}$ and $\image \left.\det\right|_G$ is of finite order.
\end{prop}

\begin{proof}
    For the first part of the necessity condition, we proceed by contradiction. Suppose there is $x\in H$ such that $x$ is not the identity. Then the order of $x$ is $p^r$ for some $r>0$, and we may assume $r=1$. Therefore, $H$ contains a cyclic subgroup of order $p$, which we denote as $P$. Since $d(G, V)\le d(P, V)$, it follows that $d(P, V)=\dim V$ by assumption. By Theorem \ref{d(P, V)}, $P$ acts on $V$ trivially but that gives a contradiction to the construction of $P$.

    The claim that $\image \left.\det\right|_G$ is finite follows from Theorem \ref{chi infinite}.

\end{proof}

Before proving the converse, we first discuss the number of connected components of $G$ under the assumption that $H=\{e\}$.

\begin{prop} \label{orderprimetop}
    Suppose $G$ is an algebraic group such that the only unipotent element is the identity. Let $G^\circ$ be the connected component of $G$ that contains the identity. Then $G/G^\circ$ is a finite group of order prime to $p$.
\end{prop}

\begin{proof}
    Consider a Borel subgroup $B$ of $G^\circ$. Since all elements of $G^\circ$ are semisimple, by \cite[Corollary 11.5 (1)]{borel2012linear}, $G^\circ$ is a torus. We have that $G/G^\circ$ is a finite group. It remains to show that $G/G^\circ$ is of order prime to $p$. 

    We prove it by contradiction.
    Write $G^\circ =T$. Suppose $p \mid \lvert G/T \rvert$. Then there exist $h_0\in G \setminus T$ such that $h_0^p \in T$. Denote the irreducible component of $G$ containing $h_0$ by $H$. Consider a morphism of varieties $f:H \to T$ sending $h$ to $h^p$.

    For any $h \in H$, write $h=h_0 t$ for some $t \in T$. Therefore, $$h^p=h_0t h_0 t \cdots h_0 t h_0 t.$$ Since $th_0=h_0 h_0^{-1}th_0=h_0t^{h_0}$ where $t^{h_0}=h_0^{-1} t h_0$,  \begin{align*}
        h^p&=h_0t h_0 t \cdots h_0 (t h_0) t=h_0t h_0 t \cdots h_0 (h_0t^{h_0}) t\\
        &=h_0h_0t h_0 t \cdots h_0 h_0t^{h_0^2}t^{h_0} t\\
        &=\cdots= h_0^p t^{h_0^{p-1}}\cdots t^{h_0^{p-2}} t^{h_0} t.
    \end{align*}

    Denote $t^{h_0^{p-1}}\cdots t^{h_0^{p-2}} t^{h_0} t$ by $t^{h_0^{p-1}+h_0^{p-2}+\cdots + h_0+1}$.

    Suppose $T$ is a torus of rank $r$. Since $T\cong \Gm ^{\X r}$ and $\Hom(\Gm, \Gm)\cong \Z$, it follows that $\Hom(T, T)\cong M_r(\Z).$ 
    
    Since $T \trianglelefteq G$, every $x\in G$ acts on $T$ by conjugation, giving a matrix in $M_r(\Z)$. Let $A$ be a matrix representing the automorphism of conjugation by $h_0$ on $T$. Since $h_0^p \in T$, it follows that $A^p=I$. Hence $A^p-I=0$ and the minimal polynomial of $A$ divides the polynomial $$x^p-1=(x-1)(x^{p-1}+ \cdots +1).$$

    Consider $A: \Q^r \to \Q^r$. Then $\Q^r=V \oplus W$ where $V=\ker(A-I)$ and $W=\ker (A^{p-1}+\cdots + I)$. Hence, $$\image(A^{p-1}+\cdots + I)=(A^{p-1}+\cdots + I)V=V=\ker (A-I).$$

    Let $S$ be the set in $T$ corresponding to $\ker(A-I)$, i.e., $$S=\{t=(t_1, t_2, \cdots, t_r) \mid t^{h_0}=t\}.$$ Hence, $S$ is the centralizer of $h_0$ in $T$, and we denote it by $C_T(h_0)$. By assumption, the identity is the only unipotent element, hence $h_0^p \neq e$. Since $h_0^p \in T$ commutes with $h_0$, it follows that $h_0^p \in C_T(h_0) \supsetneq \{e\}.$

    Since $\image f=h_0^pC_T(h_0)$ and $\image f$ is the image of an irreducible component, $C_T(h_0)$ is connected. Since $C_T(h_0)$ is a closed, connected subgroup of a torus, it is a torus. The rank of $C_T(h_0)$ is larger than zero since it contains at least two elements.

    We have $h_0^p\in C_T(h_0)$ and as $C_T(h_0)$ is a torus, there exists $t_0\in C_T(h_0)$ such that $t_0^p=h_0^p$. Since $t_0$ commutes with $h_0$,  $$(h_0 t_0^{-1})^p=h_0^p (t_0^{-1})^p=h_0^p(h_0^p)^{-1}=e.$$ Therefore, $h_0t_0^{-1}$ is a unipotent element other than the identity, which contradicts the assumption on $G$. Hence, $p \nmid \lvert G/T \rvert$.

\end{proof}

\begin{thm} \label{chi finite first time}
    If $H=\{e\}$ and $\image \left.\det\right|_G$ is of finite order, then $$d(G, V)=m.$$
\end{thm}

\begin{proof}
    By the definition of $d(G, V)$, for any $k\in \N$, $$d(G, V)\ge d(G, V^{\xx k})^{1/k}.$$ On the other hand, since $\TS_G(V^{\xx nk}) \ge \TS_G(V^{\xx n})^k$, it follows that $$d(G, V^{\xx})^{1/k} \ge d(G, V).$$ Hence, \[d(G, V)=d(G, V^{\xx k})^{1/k}.\]

    Suppose $T:V \to V$, $S: W \to W$ are two linear transformations and $\dim V=n$, $\dim W=m$. Then $$\det (T \xx S)= \det (T\xx \id_W \circ \id_V \xx S)=\det (T\xx \id_W)  \det (\id_V\xx S).$$ Since one may represent $T\xx \id_W$ as a matrix of $m$ blocks on the diagonal where each block is a matrix representing $T$, it follows that $$\det (T\xx \id_W)=\det (T)^m.$$ Therefore, $\det(T \xx S)= \det(T)^m  \det(S)^n$.

    Now suppose $\image \det|_G $ is of order $k$. Then for any $g$, $\det (g)$ is a $k$-th root of unity. 
    \begin{align*}
        \det \bigl(g^{\xx k}\bigr)&= \det \bigl(g \xx \id_{V^{\xx (k-1)}}\bigr) \cdots \det\bigl(\id_{V^{\xx (k-1)}} \xx g\bigr)\\& =\Bigl(\det(g)^{m^{k-1}}\Bigr)^k=\bigl(\det (g)^k\bigr)^{m^{k-1}}=1.
    \end{align*}

    Therefore, $\rho^{\xx k}:G \to \SL\bigl(V^{\xx k}\bigr) \subset \GL\bigl(V^{\xx k}\bigr)$. Since $d(G, V)=d\bigl(G, V^{\xx k}\bigr)^{1/k}$, $d(G, V)=m$ is equivalent to $d\bigl(G, V^{\xx k}\bigr)=m^{k}$. Now we may assume $G \subset \SL(V)$ and the only unipotent element of $G$ is the identity, and we want to show $d(G, V)= \dim V=m$.   

    Let $G^\circ$ be the connected component of $G$ containing the identity. Since all elements of $G^\circ$ are semisimple, $G^\circ$ is a torus, and we denote it by $T$. Decompose $V^{\xx mn}$ into a direct sum (as vector spaces) of $T$-eigenspaces $E_\l$. Since $T$ is a normal subgroup of $G$,  $G$ permutes $E_\l$ because $tg=gt^g$ where $t^g=g^{-1}tg$ and for $v\in E_\l$, $gv \in E_{\l^g}$ where $\l^g(t)=\l(t^g)$ since $t(gv)=gt^gv=g\l(t^g)v=\l(t^g)(gv)$. 
    
    For $\l \neq 0 \in X(T)$, $gE_\l \neq E_0$ by the difference of the dimensions of the kernel of $\l^g$ and the kernel of the trivial character $0$. Hence, $E_0$ is $G$-invariant and a direct summand of $V^{\xx mn}$, and it induces a representation of $\Sigma \coloneq G/T$. 

    Consider a general situation for the moment. Let $W$ be an $m$-dimensional vector space over an algebraically closed field $F$ of any \ch . Let $D$ be the group of diagonal matrices in $\SL(W)$. The number of $\TS_D\bigl(W^{\xx nm}\bigr)$ is equal to the dimension of the weight zero subspace of $W^{\xx mn}$. This number is the number of $mn$-tuple $(a_1, \cdots, a_{mn})$ where $a_i \in \{1, \cdots, m\}$ such that $(n_1, \cdots, n_m)=(n,\cdots, n)$ where $n_i$ is the number of times $i$ occurs in $\{a_k\}_{k=1}^{mn}$. This number is a purely combinatorial answer so it does not depend on the \ch\ of $F$. Therefore, we are free to assume $F$ is an algebraically closed field of any characteristic, and compare the number $\TS_D\bigl(W^{\xx nm}\bigr)$ with the numbers of trivial summands corresponding to different groups under different characteristic settings.

    Back to our setting of Theorem \ref{chi finite first time}, let $D$ be the maximal torus consisting of the diagonal matrices in $\SL(V)$. Denote the weight zero subspace of $V^{\xx mn}$ corresponding to $D$ by $W_{n}$. We may assume $D \supset T$; then the subspace of $V^{\xx mn}$ fixed by $T$ contains $W_{n}$, and we denote it by $\widetilde{W}_{n}$. Since $\widetilde{W}_{n}$ is $G$-invariant and every element of $\widetilde{W}_{n}$ is fixed by $T$,  $\widetilde{W}_{n}$ is a $\Sigma$-representation. By Proposition \ref{orderprimetop}, $p \nmid \lvert \Sigma \rvert$; hence, $\widetilde{W}_{n}$ is a semisimple $\Sigma$-representation, therefore, a semisimple $G$-representation. Hence, any $\Sigma$-subrepresentation in $\widetilde{W}_{n}$ is a direct summand in $V^{\xx mn}$ as a $G$-representation.

    By the definition of $\widetilde{W}_k$, $$\widetilde{W}_k\xx \widetilde{W}_l \subset \widetilde{W}_{k+l}\subset V^{\xx m(k+l)}.$$

     If we find a sequence $(n_k)_k$ such that $\widetilde{W}_{n_k}$ contains enough trivial $\Sigma$-representations, then the result is proved. Since $\Sigma$ is a finite group, there are only finitely many different normal subgroups. Hence, after taking a finite quotient of $\Sigma$, there exists a sequence $(n_k)_k$ such that $\rho_k: \Sigma \to \GL\Bigl(\widetilde{W}_{n_k}\Bigr)$ is a faithful $\Sigma$-representation. 

     If $\Sigma$ is the trivial group, then it is done since $\TS_\Sigma\Bigl(\widetilde{W}_{n_k}\Bigr)=\dim \widetilde{W}_{n_k}$ and      
     $$\TS_G(V^{\xx mn_k}) =\TS_\Sigma(\widetilde{W}_{n_k})=  \dim \widetilde{W}_{n_k} \ge \dim W_{n_k}\ge \TS_{\SL(V_\C)}(V_\C^{\xx mn_k}),$$ where $V_\C$ is an $m$-dimensional vector space over $\C$, hence $d(G, V)=m$.

     Suppose $\Sigma$ is not trivial. Since $\Sigma$ is of order prime to $p$ and $K$ is algebraically closed, by \cite[Proposition 43]{serre1977linear} $R_K(\Sigma)=R_\C(\Sigma)$, where $R_L(\Sigma)$ is the Grothendieck ring of $\Sigma$ over the field $L$. Therefore, we may assume $\widetilde{W}_{n_k}$ is a representation of $\Sigma$ over $\C$.

    Consider the sequence $(\widetilde{W}_{n_k})_k$. For each $\widetilde{W}_{n_k}$, by Proposition \ref{faithproduceseverything}, there exists $d_k$ such that $\widetilde{W}_{n_k}^{\xx d_k}$ contains a trivial representation as a $\Sigma$-representation. 

    If there are infinitely many $k$ such that $\widetilde{W}_{n_k}^{\xx d_k}$ is a direct sum of trivial representations, then it is done for the following reason: take a subsequence of $(n_kd_k)_k$, say $(n_{k_i}d_{k_i})_i$, such that $n_{k_i} d_{k_i} >n_{k_j} d_{k_j}$ and $n_{k_i}>n_{k_j}$ whenever $i>j.$ Consider the sequence $(V^{\xx mn_{k_i} d_{k_i}})_i$. Since $$\widetilde{W}_{n_k}^{\xx d} \subset \widetilde{W}_{dn_k}\subset V^{\xx mdn_k},$$ it follows that \begin{align*}
        \TS_G\bigl(V^{\xx mn_{k_i} d_{k_i}}\bigr)= \TS_\Sigma \Bigl(\widetilde{W}_{{n_k}_i {d_k}_i}\Bigr) \ge \TS_\Sigma \Bigl(\widetilde{W}_{{n_k}_i}^{\xx {d_k}_i}\Bigr)=\Bigl(\dim \widetilde{W}_{{n_k}_i}\Bigr)^{{d_k}_i},
    \end{align*} 
    where the last equality follows from the condition that $\widetilde{W}_{n_k}^{\xx d_k}$ is a direct sum of trivial representations. Therefore, 
    
    \begin{align*}
    \limsup_{i \to \infty}\TS_G\bigl(V^{\xx mn_{k_i}d_{k_i}}\bigr)^{1/mn_{k_i} d_{k_i}} &\ge \limsup_{i \to \infty}\Bigl(\dim \widetilde{W}_{{n_k}_i}\Bigr)^{{{d_k}_i}/mn_{k_i} d_{k_i}}\\
    =\limsup_{i \to \infty}\bigl(\dim \widetilde{W}_{{n_k}_i}\bigr)^{1/mn_{k_i}} &\ge \limsup_{i \to \infty}\bigl(\dim W_{{n_k}_i}\bigr)^{1/mn_{k_i}}=m.
\end{align*}

    Otherwise, there exist infinitely many $k$ such that $\widetilde{W}_{n_k}^{\xx d_k}$ contains a trivial summand but itself is not a direct sum of trivial representations. Since $\Sigma$ is a finite group, there exists a normal subgroup $N$ and a subsequence $(n_{k_l})_l$ such that $N$ is the kernel of $\pi_l: \Sigma \to \GL\Bigl(\widetilde{W}_{n_{k_l}}^{\xx d_{k_l}}\Bigr)$. Therefore the induced $\pi_l: \Sigma / N \to \GL\Bigl(\widetilde{W}_{n_{k_l}}^{\xx d_{k_l}}\Bigr)$ is a faithful $\Sigma/N$-representation. 

    Replace $\Sigma$ with $\Sigma/N$ and $n_l$ with $n_{k_l}$. By Lemma \ref{producesregularrep}, there exists some $M$ such that  $\Bigl(\widetilde{W}_{n_1}^{\xx d_1}\Bigr)^{\xx M}$ contains the regular representation $V_\Sigma$ of $\Sigma$. Since $\widetilde{W}_{n_k}^{\xx d_k}$ is a $\Sigma$-representation, $\Bigl(\widetilde{W}_{n_1}^{\xx d_1}\Bigr)^{\xx M}\xx \widetilde{W}_{n_k}^{\xx d_k} \supset V_\Sigma \xx \widetilde{W}_{n_k}^{\xx d_k}=V_\Sigma^{\oplus c_k^ {d_k}}$ where $c_k=\dim \widetilde{W}_{n_k}$ by Remark \ref{tensorwithreg}.  
    
    Consider the sequence $\Bigl(V^{\xx mn_1d_1M}\xx V^{\xx m n_kd_k}\Bigr)_k$. Then $$V^{\xx mn_1d_1M}\xx V^{\xx m n_kd_k} \supset \Bigl(\widetilde{W}_{n_1}^{\xx d_1}\Bigr)^{\xx M}\xx \widetilde{W}_{n_k}^{\xx d_k} \supset V_\Sigma^{\oplus c_k^ {d_k}}.$$ 
    
    Then $\TS_G\Bigl(V^{\xx m(n_1d_1M+n_k d_k)}\Bigr)\ge  \Bigl(\dim \widetilde{W}_{n_k}\Bigr)^{d_k} \ge (\dim W_{n_k})^{d_k}$. Since $\dim W_{n_k}$ equals the dimension of the weight zero subspace of $V_\C^{\xx m n_k}$ corresponding to the diagonal subgroup of $\SL(V_\C)$, which we denote by $D'$,
    \begin{align*}
    &\TS_G\Bigl(V^{\xx m(n_1d_1M+n_k d_k)}\Bigr)^{1/m(n_1d_1M+n_k d_k)}\\ &\ge \TS_{D'}\bigl(V_\C^{\xx mn_k}\bigr)^{d_k/m(n_1d_1M+n_k d_k)} \\ &\ge \TS_{\SL(V_\C)}(V_\C^{\xx m n_k})^{d_k/m(n_1d_1M+n_k d_k)}.
    \end{align*} Since $n_1d_1M$ is fixed, 
    
    \begin{align*}
    &\limsup_{k \to \infty}\TS_G\Bigl(V^{\xx m(n_1d_1M+n_k d_k)}\Bigr)^{1/m(n_1d_1M+n_k d_k)}\\ &\ge \limsup_{k \to \infty} \TS_{\SL(V_\C)}\bigl(V_\C^{\xx m n_k}\bigr)^{d_k/m(n_1d_1M+n_k d_k)}\\
    &=\limsup_{k \to \infty} \TS_{\SL(V_\C)}\bigl(V_\C ^{\xx m n_k}\bigr)^{1/mn_k} \\ &= m.  
    \end{align*}

\end{proof}

Theorem \ref{main_char_p} follows from Proposition \ref{char_p_easy_direction},  Theorem \ref{chi finite first time}, and the observation used when proving Theorem \ref{main_char_zero}.

\begin{appendices}
\section{Decompositions of tensor products of modular representations of the cyclic group of order $p$} \label{appendix: decomp}

Let $p$ be an odd prime. 
In this section, we give a formula for the decomposition of tensor products of modular representations of $\Z/p\Z$.

Let $V_1=\text{Span}_k (x, y)$ be the indecomposable representation of dimension $2$ where $g\cdot x=x,\ g\cdot y=x+y$. Denote $\text{Sym}^i\, V_1$ by $V_i$. Then $V_i$ is the unique indecomposable representation of dimension $i+1$ up to isomorphism for $0 \le i \le p-1.$

\begin{prop} \label{decomVp-1}
    $$V_{p-1}\xx V_i=(i+1)V_{p-1}.$$
\end{prop}
\begin{proof}
    Since $V_{p-1}$ is a free $k[G]$-module, it is projective. By \cite[Section 14.2]{serre1977linear}, $V_{p-1}\xx V_i$ is also projective. Since every indecomposable projective $k[G]$-module is isomorphic to $V_{p-1}$, it follows that $V_{p-1}\xx V_i \cong (i+1)V_{p-1}.$
\end{proof}

\begin{prop} \label{decomV1}
\[V_d\xx V_1 =\begin{cases}
    V_1,\ d=0\\
    V_{d+1}\oplus V_{d-1},\ 0<d<p-1\\
    2V_{p-1},\ d=p-1.
\end{cases} \]
\end{prop}

\begin{proof}
For $d=0$, it holds since $V_0$ is the trivial representation.

For $d=p-1$, this follows from Proposition \ref{decomVp-1}.

For $0<d<p-1$, write $V_d=\text{Span}_K(x^d, x^{d-1}y, x^{d-2}y^2, \cdots, y^d).$ Then \begin{align*}
    V_d \xx V_1=\text{Span}_K (&x^d\xx x, x^{d-1}y\xx x, \cdots, y^d \xx x, \\ 
&x^d \xx y, x^{d-1}\xx y, \cdots, y^d \xx y).
\end{align*}

Let $W_l$ be the subspace with the basis $\{x^{d-i}y^i \xx x^{1-j}y^j \mid i+j\le l\}$ and $W_{-1}=\{0\}$.  Then $W_{-1} \subset W_0 \subset W_1 \subset \cdots \subset W_d \subset W_{d+1}=V_d \xx V_1$. Let $T$ be the action of $(g-e)$ on $V_d \xx V_1$. Then $T(W_l)\subset W_{l-1}$ by direct computation and $T^{d+2}=0$.

For the later use, we have \begin{equation}\label{1}
        \begin{split}
            T(x^{d-l}y^l \xx x) &\equiv lx^{d-l+1} y^{l-1} \xx x \qquad \text{(mod $W_{l-2}$}), \\
     T(x^{d-l+1}y^{l-1} \xx y) &\equiv x^{d-l+1} y^{l-1} \xx x + (l-1) x^{d-l+2} y^{l-2} \xx y  \quad \text{(mod $W_{l-2}$)}.
        \end{split}
    \end{equation}

Showing $V_{d+1}$ is a direct summand of $V_d\xx V_1$ is equivalent to showing that one of the Jordan blocks is of size $d+2$. Since $T^{d+2}=0$, all sizes of Jordan blocks of $T$ are less than or equal to $d+2$. Hence, it suffices to find a vector $v$ such that $T^{d+1}(v)\neq 0$. Let $v=y^d\xx y$. We first use induction and (\ref{1}) to show that
\begin{equation} \label{2}
        \begin{split}
         T^k (v) \equiv kd(d-1)(d-2)\cdots (d-k+2)x^{k-1}y^{d-k+1}\xx x \\ + d(d-1)(d-2)\cdots (d-k+1)x^k y^{d-k}\xx y \quad \text{(mod $W_{d-k}$)}    
        \end{split}
\end{equation}

when $d \ge k \ge 2$, and the result follows from applying $T$ to $T^d(v)$.

Since $T$ sends $W_l$ to $W_{l-1}$, $T$ sends $W_l/W_{l-2}$ to $W_{l-1}/W_{l-3}$.

By (\ref{1}), we have $[T(v)]= [y^d\xx x + dxy^{d-1}\xx y] \in W_d/W_{d-1}$.

Hence, \begin{align*}
    [T^2 (v)]&=[T(y^d\xx x + dxy^{d-1}\xx y)]\\ &=[dxy^{d-1}\xx x + d(d-1)x^2 y^{d-2}\xx y +dxy^{d-1}\xx x]\\&=[2dxy^{d-1}\xx x + d(d-1)x^2y^{d-2}\xx y] \in W_{d-1}/W_{d-2}.
\end{align*}

Suppose the formula (\ref{2}) holds up to $k=n$.

For $k=n+1$, \begin{align*}
    [T^{n+1}(v)]&=[T(T^n(v))]    \\&=[T(nd(d-1)\cdots (d-n+2)x^{n-1}y^{d-n+1}\xx x\\ &+ d(d-1)\cdots (d-n+1)x^n y^{d-n}\xx y)]\\&=[nd(d-1)\cdots(d-n+1)x^ny^{d-n}\xx x\\ &+(d(d-1)\cdots (d-n+1)x^ny^{d-n}\xx x\\ &+d(d-1)\cdots (d-n) x^{n+1}y^{d-n-1}\xx y)]\\&=[(n+1)d(d-1)\cdots(d-n+1) x^ny^{d-n}\xx x \\ &+d(d-1)\cdots (d-n)x^{n+1}y^{d-n-1} \xx y)] \in W_{d-n}/W_{d-n-1}.
\end{align*} 
Hence, the formula holds up to $k=d$.

Then, $T^d(v)=d\cdot d!\, x^{d-1}y\xx x+d!\, x^d \xx y +cx^d \xx x$ for some $c$, and $T^{d+1}(v)=d!\, x^d\xx x +d d!\, x^d\xx x=(d+1)!\, x^d\xx x$. Therefore, $T^{d+1}(v) \neq 0$ when $d<p-1$.

The set $\{v, T(v), \cdots, T^{d+1}(v)\}$ determines a Jordan block of size $d+2$. Let $W$ be the space spanned by these vectors. Write $V_d\xx V_1=W \oplus U$. Then $(V_d\xx V_1) /W \cong U$. 

Let $u=y^d\xx x$. Since $g\cdot x= x$ and $$T^i(u)\equiv d(d-1)\cdots (d-i+1) x^i y^{d-i}\xx x \quad \text{(mod $W_{d-i-1}$)},$$ one has $$T^i(u)=d(d-1)\cdots (d-i+1) x^i y^{d-i}\xx x + \sum_{j>i}a_j x^j y^{d-j}\xx x$$ and $T^d(u)=d!\, x^d \xx x$. Comparing $T^i(u)$ and the equation $(\ref{2})$, one has that the set $\{[u], [T(u)], \cdots [T^{d-1}(u)]\}$ in $(V_d\xx V_1)/W$ is linear independent. Therefore, $U \cong V_{d-1}$ and $V_d\xx V_1 \cong V_{d+1}\oplus V_{d-1}$.

\end{proof}

\begin{prop} \label{decomVm}
    \[V_m \xx V_n=
    \begin{cases}
        V_{\lvert m-n \rvert} \oplus V_{\lvert m-n+2 \rvert} \oplus \cdots \oplus V_{\lvert m+n \rvert},\ \text{when } m+i \le p-1, \\

        dV_{p-1} \oplus V_{m-d}\xx V_{n-d},\ \text{when } m+n=p-2+d \ \text{and}\ d \ge 1.
    \end{cases}\]
\end{prop}

\begin{proof}
    The formula for $V_m \xx V_n$ breaks into cases (a) $0 \le m+n \le p-1$ and (b) $m+n=p-2+d$ and $d\ge 1$. 

    For (a):
    first we assume $m \ge n$ and prove by induction on $n$.

    When $n=0$, it holds.
    Suppose the formula for (a) holds up to $n=k$. For $n=k+1$ and suppose $m+k+1 \le p-1$  since $V_n=V_{n-1} \xx V_1 - V_{n-2}$ by Proposition \ref{decomV1},

    \begin{align*}
        V_m \xx V_n &= V_m \xx (V_{n-1} \xx V_1 - V_{n-2})\\
        &=\begin{multlined}[t]
            (V_{ m-n+1 } \oplus V_{ m-n+3 } \oplus \cdots \oplus V_{ m+n-1 })\xx V_1 \\
            -(V_{ m-n+2 } \oplus V_{ m-n+4 } \oplus \cdots \oplus V_{ m+n-2 })
        \end{multlined} \\
        &=V_{ m-n } \oplus V_{ m-n+2 } \oplus \cdots \oplus V_{ m+n }.
    \end{align*}
      Exchange $n$ and $m$ when $m \le n$, then the formula for (a) holds.

   For (b), fix $m$, we prove by induction on $d$.

   When $d=1$, applying the formula for (a), one has \[V_m \xx V_n=
        V_{\lvert m-n \rvert} \oplus V_{\lvert m-n+2 \rvert} \oplus \cdots \oplus V_{\lvert m+n-2 \rvert}\oplus V_{\lvert m+n \rvert}=V_{m-1}\xx V_{n-1}\oplus V_{p-1}.\]

    Suppose the formula for (b) holds up to $d=k$.

    When $d=k+1$, then \begin{align*}
        V_m \xx V_n &=V_m \xx (V_{n-1} \xx V_1 - V_{n-2})\\
        &=((d-1)V_{p-1}\oplus V_{m-(d-1)} \xx V_{n-d})\xx V_1-((d-2)V_{p-1} \oplus V_{m-(d-2)}\xx V_{n-d})    \\
        &=(2d-2)V_{p-1} \oplus V_{(m-d+1)+1}\xx V_{n-d}\oplus V_{(m-d+1)-1}\xx V_{n-d}\\ &-((d-2)V_{p-1} \oplus V_{m-(d-2)}\xx V_{n-d})\\
        &=dV_{p-1}\oplus V_{m-d}\xx V_{n-d}.
    \end{align*}
    
\end{proof}

\section{Some observations on regular representations} \label{appendix: regular}

Let $G$ be a finite group. By \cite[Exercise 2.37]{fulton2013representation}, we have the following proposition. 
\begin{prop} \label{faithproduceseverything}
Let $\rho: G \to \GL(V) $ be a faithful representation over $\C$. Then every irreducible representation of $G$ is a direct summand of $V^{\xx d}$ for some $d$ depending on the irreducible representation.    
\end{prop}

\begin{lemma} \label{producesregularrep}
    Suppose $V$ is a faithful representation of $G$. Let $\1$ be the trivial representation of $G$. Then there exists some integer $N$ such that $(\mathbbm{1} \oplus V)^k$ contains the regular representation as a direct summand for all $k \ge N$.
\end{lemma}

\begin{proof}
    Let $V_0, V_1, V_2, \cdots, V_n$ be irreducible representations of $G$ up to isomorphism. Let $V_0$ be the trivial representation.
    Let $n_i$ be a number such that $V^{\xx n_i}$ contains $V_i$ as a direct summand. Then $V^{\xx n_i+l n_0}$ contains $V_i$ as a direct summand for any nonnegative integer $l$. 
    
    Since $$(\1 \oplus V)^{\xx N}=\1 \oplus a_1 V \oplus a_2V^{\xx 2} \oplus \cdots \oplus a_{N-1}V^{\xx N-1} \oplus V^{\xx N}$$ where $a_i= \binom{N}{i} \neq 0$ for any $i$, there exists some integer $N$ such that $(\1 \oplus V)^N$ contains a regular representation. Therefore, $(\1 \oplus V)^{\xx k}$ contains a regular representation as a direct summand for all $k>N$.
\end{proof}

\begin{lemma}
    Let $V$ be a representation of $G$ of degree $d$ and $V_G$ be the regular representation of $G$. Then $V\xx _{\C} V_G \cong V_G^{\oplus d}$.
\end{lemma}

\begin{proof}
    Let $\chi$ be the character of $V\xx _{\C} V_G$. Then $\chi(g)=0$ when $g\neq e$ and $\chi(e)=d \vert G \vert $. Hence, $\chi=d\chi_G$ where $\chi_G$ is the character of the regular representation of $G$. Since a representation is determined by its character in the \ch\ zero case, the result follows. 
\end{proof}

\begin{remark} \label{tensorwithreg}
    Since $V_G$ contains exactly one trivial representation, for any $G$-representation $V$ of degree $d$, we have $\TS_G(V\xx V_G)=\TS_G\bigl(V_G^{\oplus d}\bigr)=d$.
\end{remark}

\end{appendices}

\printbibliography
\end{document}